\documentclass[a4paper]{article}

\usepackage[14pt]{extsizes}
\usepackage[left=25mm, top=25mm, right=15mm, bottom=25mm, nohead, footskip=10mm]{geometry}
\usepackage[utf8]{inputenc}
\usepackage[english]{babel}
\usepackage{amsfonts,amsmath,amsthm,amssymb}
\usepackage{cmap}
\usepackage[usenames]{color}
\usepackage{enumitem}

\newtheorem{theorem}{Theorem}[section]
\newtheorem{remark}{Remark}[section]
\newtheorem{lemma}{Lemma}[section]
\newtheorem{corollary}{Corollary}[section]
\newtheorem{example}{Example}[section]

\newcommand\R{\mathbb R}
\newcommand\N{\mathbb N}
\newcommand\E{\mathbb E}

\newcommand\F{\mathcal F}
\newcommand\B{\mathcal B}

\newcommand\const{\mathrm{const} \, }
\newcommand\T{\mathrm{T}}
\newcommand\radial{_\mathrm{rad}}
\newcommand\tangential{_\mathrm{tan}}
\newcommand\sgn{\mathrm{sgn} \, }
\newcommand\C{\mathrm{C}}

\renewcommand\d{\mathrm d}

\renewcommand\P{\mathbb P}

\begin{document}
	
\title{On Asymptotic Behavior of Stochastic Differential Equation Solutions in Multidimensional Space}
\author{Viktor Yuskovych}
\maketitle
	
\begin{abstract}
Consider the multidimensional SDE
$$\d X(t) = a(X(t))\d t + b(X(t))\d W(t).$$
We study the asymptotic behavior of its solution $X(t)$ as $t \to \infty$, namely, we study sufficient conditions of transience of its solution $X(t)$, stabilization of its multidimensional angle $X(t)/|X(t)|$, and asymptotic equivalence of solutions of the given SDE and the following ODE without noise:
$$\d x(t) = a(x(t))\d t.$$
\end{abstract}	
	
\section{Introduction}

Usually, there are two modes of behavior of SDE solutions as $t \to \infty$: transience and recurrence. In this article, we study transience of solutions.

Consider a one-dimensional SDE of the form
\begin{equation}\label{sde}
	\d X(t) = a(X(t))\d t + b(X(t))\d W(t).
\end{equation}
Gikhman and Skorokhod (see § 16, 17 of Part I in \cite{gichman}) were the first who started studying its non-random asymptotics, i.e., a function $x(t)$ such that $X(t) \sim x(t), \ t \to \infty,$ a.s. Later this problem was studied by Keller, Kersting, and Rösler \cite{keller}. Buldygin, Indlekofer, Klesov, Stainebach, and Tymoshenko (see \cite{buldygin1}, \cite{buldygin2}) considered some types of non-autonomous SDEs and studied asymptotic behavior of their solutions; in particular, they considered the problem of asymptotic equivalence of SDE and ODE solutions. Pilipenko, Proske, and Pavlyukevich (see \cite{pilipenko1}, \cite{pilipenko2}) considered SDEs with a non-Gaussian noise.

Unlike the one-dimensional case, asymptotic behavior of the multidimensional SDE solution differs even provided its transience. Friedman \cite{friedman} and Khasminskii \cite{khasminskii} studied conditions of transience and recurrence for systems of linear SDEs. Friedman also studied the behavior of the polar angle of the two-dimensional SDE solution (see § 12.7 in \cite{friedman}). Spitzer \cite{spitzer} studied the limit distribution (as $t \to \infty$) of the polar angle of the Wiener process on the plane. Samoilenko, Stanzhitskii, Novak \cite{samoilenko} and Pilipenko, Proske \cite{pilipenko3} studied transience of solutions to multidimensional SDEs.

Consider an $n$-dimensional ($n \geqslant 2$) autonomous SDE of the form
\begin{equation}\label{sde}
	\d X(t) = a(X(t))\d t + b(X(t))\d W(t), \ X(0) = x_0 \in \R^n,
\end{equation}
where $a\colon \R^n \to \R^n$, $b\colon \R^n \to \R^{n \times m}$, and $W$ is an  $m$-dimensional Wiener process

In this article, we study asymptotic behavior of solutions $X(t)$ of the SDE (\ref{sde}) as $t \to \infty$. Namely, we search for sufficient conditions such that:
\begin{itemize}
	\item the solution $X(t)$ is transient, i.e., almost surely
	$$|X(t)| \to \infty, \ t \to \infty;$$
	\item the angle of the solution's growth stabilizes, i.e., there exists a random variable $\Phi_\infty$ ({\it the limit angle}) such that the limit
	$$\lim_{t \to \infty} \frac{X(t)}{|X(t)|} =: \Phi_\infty$$
	exists almost surely;
	\item there exists a non-random function of two variables $r_\varphi(t), \ \varphi \in \R^n, \ t \geqslant 0$, that describes the asymptotics of $|X(t)|$, i.e., almost surely
	$$|X(t)| \sim r_{\Phi_\infty}(t), \ t \to \infty,$$
	where $\Phi_\infty$ is the limit angle.
\end{itemize}

For convenience, denote $|X(t)| =: R(t)$ and $X(t)/|X(t)| =: \Phi(t)$. We will call $R(t)$ the {\it radius} process and $\Phi(t)$ the {\it angle} process.

It is known that if the diffusion is non-degenerate ($\det b^\T b > 0$) then the solution of a multidimensional ($n \geqslant 2$) SDE starting at $x_0 \neq 0$ never hits zero with probability 1. Without loss of generality, we will assume that $X(0)~ =~x_0~\neq~0$. Applying Itô formula to the radius and angle processes, we get 
\begin{eqnarray}
	\label{sde_radius}\d R(t) = \mu(R(t), \Phi(t))\d t + \boldsymbol\sigma(R(t), \Phi(t))\d W(t), \ R(0) = r_0 = |x_0| > 0, \\
	\label{sde_angle}\d\Phi(t) = \nu(R(t), \Phi(t))\d t + \chi(R(t), \Phi(t))\d W(t), \ \Phi(0) = \varphi_0 = \frac{x_0}{|x_0|},
\end{eqnarray}
where $\mu\colon \R\times\R^n \to \R$, $\boldsymbol\sigma\colon \R\times\R^n \to \R^m$, $\nu\colon \R\times\R^n \to \R^n$, $\chi\colon \R\times\R^n \to \R^m$ are some functions.

Let's write down formulae for coefficients $\mu$, $\boldsymbol\sigma$, $\nu$, $\chi$ of equations (\ref{sde_radius})-(\ref{sde_angle}) in terms of coefficients $a$, $b$ of the initial SDE (\ref{sde}). For this, define the {\it radial} and {\it tangential} components of the vector field $a$ at the point $x \neq 0$ by
$$a\radial(x) := \frac{xx^\T}{|x|^2}a(x), \qquad a\tangential(x) := a(x) - a\radial(x),$$
respectively. Similarly, define the radial and tangential components of the matrix field $b$ at the point $x \neq 0$ by
$$b\radial(x) := \frac{xx^\T}{|x|^2}b(x), \qquad b\tangential(x) := b(x) - b\radial(x),$$
respectively. Then
$$\mu(r, \varphi) = \varphi^\T a\radial(r\varphi) + \frac{|b\tangential(r\varphi)|^2}{2r}, \qquad \boldsymbol\sigma(r, \varphi) = \varphi^\T b\radial(r\varphi),$$
$$\nu(r, \varphi) = \frac{a\tangential(r\varphi)}r-\frac{\left(2(b(r\varphi)b^\T(r\varphi))\tangential\right) + |b\tangential(r\varphi)|^2)\varphi}{2r^2},$$
$$\chi(r, \varphi) = \frac{b\tangential(r\varphi)}{r}.$$

Using the Lévy theorem (see § 7 of Chapter II in \cite{ikeda}), one can find a one-dimensional Wiener process $W^{(1)}$ and a function $\sigma\colon \R\times\R^n \to \R$ such that
$$\boldsymbol\sigma(R(t), \Phi(t))\d W(t) = \sigma(R(t), \Phi(t))\d W^{(1)}(t).$$

Hence, the SDE for the radius process can be written as follows:
$$\d R(t) = \mu(R(t), \Phi(t))\d t + \sigma(R(t), \Phi(t))\d W^{(1)}(t).$$

Further, we focus on studying the system of SDEs 
\begin{eqnarray}
	\label{sde_radius_new} \d R(t) = \mu(R(t), \Phi(t))\d t + \sigma(R(t), \Phi(t))\d W^{(1)}(t), \ R(0) = r_0, \\
	\label{sde_angle_new} \d\Phi(t) = \nu(R(t), \Phi(t))\d t + \chi(R(t), \Phi(t))\d W(t), \ \Phi(0) = \varphi_0
\end{eqnarray}
considering coefficients $\mu, \sigma, \nu, \chi$ to be arbitrary (not related to the coefficients $a, b$ of the initial SDE (\ref{sde})). Nevertheless, we will keep calling the processes $R$ and $\Phi$ the radius and the angle, respectively. Results about $R(t)$ and $\Phi(t)$ obtained below will describe the asymptotic behavior of the solution $X(t)$ to the SDE~(\ref{sde}).

Notice that the problems stated previously now can be reformulated in terms of the radius and the angle processes; namely, we search for sufficient conditions such that the following hold almost surely:
\begin{itemize}
	\item $R(t) \to \infty, \ t \to \infty$;
	\item $\exists \lim_{t \to \infty} \Phi(t) =: \Phi_\infty$;
	\item $\exists r_\varphi(t)\colon R(t) \sim r_{\Phi_\infty}(t), \ t \to \infty$.
\end{itemize}

This article has the following structure. In Section 2, we prove a general theorem about the asymptotic equivalence of SDE and ODE solutions in the one-dimensional non-autonomous case. In Section 3, we state sufficient conditions that guarantee the transience of the radius. In Section 4, we prove a theorem about angle stabilization. In Section 5, we prove the main result about radius asymptotics. Appendix contains some auxiliary results.

\section{Asymptotic Behavior of One-Dimensional SDEs}

Let $(\Omega, \F, (\F_t)_{t \geq 0}, \P)$ be a filtered probability space, $B$ be a one-dimensional $(\F_t)_t$-adapted Wiener process.

Let's prove the next lemma about asymptotics of Itô integrals.

\begin{lemma}\label{lemma_integral}
	Let $b = b(t, \omega)$ be a $(\F_t)_t$-adapted stochastic process and $C > 0$, $\beta > 0$ be constants such that
	$$\E b^2(t) \leq C(1 + t^{2\beta}), \ t \geq 0.$$
	Then for any $\gamma > \beta + \frac12$, almost surely
	$$\frac1{t^\gamma} \int_0^t b(s) \d B(s) \to 0, \ t \to \infty.$$
\end{lemma}
\begin{proof}
	Let $\varepsilon > 0$, $k \in \N_0$. Then using Doob's martingale inequality, the Itô isometry, and Fubini's theorem, we have
	$$\P\left\{\sup_{2^k \leq t \leq 2^{k + 1}} \frac1{t^\gamma}\left|\int_0^t b(s) \d B(s)\right| \geq \varepsilon\right\} \leq$$
	$$\leq \P\left\{\sup_{2^k \leq t \leq 2^{k + 1}} \frac1{(2^k)^\gamma}\left|\int_0^t b(s) \d B(s)\right| \geq \varepsilon\right\} \leq$$
	$$\leq \P\left\{\sup_{0 \leq t \leq 2^{k + 1}} \left|\int_0^t b(s) \d B(s)\right| \geq \varepsilon2^{\gamma k}\right\} \leq$$
	$$\leq \frac1{(\varepsilon 2^{\gamma k})^2}\E\left(\int_0^{2^{k+1}}b(s)\d B(s)\right)^2 = \frac1{\varepsilon^2 2^{2\gamma k}}\E\int_0^{2^{k+1}}b^2(s)\d s =$$
	$$= \frac1{\varepsilon^2 2^{2\gamma k}}\int_0^{2^{k+1}}\E b^2(s)\d s \leq \frac1{\varepsilon^2 2^{2\gamma k}}\int_0^{2^{k+1}}C(1 + s^{2\beta})\d s =$$
	$$= \frac1{\varepsilon^2 2^{2\gamma k}}\left(2^{k + 1} + \frac{(2^{k+1})^{2\beta + 1}}{2\beta + 1}\right) =$$
	$$= \frac{2C}{\varepsilon^2}2^{(1 - 2\gamma)k} + \frac{2^{2\beta + 1}C}{(2\beta + 1)\varepsilon^2}2^{(2\beta-2\gamma+1)k}.$$
	
	Hence, for any $n \in \N$, 
	$$\P\left\{\limsup_{t \to \infty}\frac1{t^\gamma}\left|\int_0^t b(s)\d B(s)\right| \geq \varepsilon\right\} \leq \P\left\{\sup_{t \geq 2^n} \frac1{t^\gamma}\left|\int_0^t b(s)\d B(s)\right| \geq \varepsilon\right\} \leq$$
	$$\leq \sum_{k = n}^\infty \P\left\{\sup_{2^k \leq t \leq 2^{k + 1}} \frac1{t^\gamma}\left|\int_0^t b(s) \d B(s)\right| \geq \varepsilon\right\} \leq$$
	$$\leq \frac{2C}{\varepsilon^2}\sum_{k = n}^\infty 2^{(1 - 2\gamma)k} + \frac{2^{2\beta + 1}C}{(2\beta + 1)\varepsilon^2}\sum_{k = n}^\infty2^{(2\beta-2\gamma+1)k}.$$
	Last two series converge to 0 as $n \to \infty$ (since $2\gamma-2\beta-1 > 0$ by condition) so the right-hand side converges to $0$ as $n \to \infty$.
	
	Since $\varepsilon > 0$ is arbitrary,
	$$\P\left\{\limsup_{t \to \infty}\frac1{t^\gamma}\left|\int_0^t b(s)\d B(s)\right| > 0\right\} = 0.$$
	Therefore,
	$$\P\left\{\limsup_{t \to \infty}\frac1{t^\gamma}\left|\int_0^t b(s)\d B(s)\right| = 0\right\} = 1,$$
	that is
	$$\P\left\{\lim_{t \to \infty}\frac1{t^\gamma}\int_0^t b(s)\d B(s) = 0\right\} = 1,$$
	and the lemma is proved.
\end{proof}

Let $X$ be a solution of the one-dimensional non-autonomous SDE
\begin{equation}
	\label{general_sde} \d X(t) = a(X(t), t, \omega)\d t + b(X(t), t, \omega)\d B(t), \qquad X(0) = x_0 \in (x_1, x_2),
\end{equation}
where $a = a(x, t, \omega), b = b(x, t, \omega)$ are $\B(\R)\times[0, t]\times\F_t$-measurable. 

The next theorem generalizes the results of Gikhman and Skorokhod (see § 17 of Part I in \cite{gichman}).

\begin{theorem}
	Suppose that:
	\begin{itemize}
		\item the coefficient $a$ is bounded and
		$$\lim_{\substack{x \to +\infty \\ t \to \infty}} a(x, t) = A \ \text{a.s.},$$
		where $A > 0$ is a positive random variable;
		\item there exist constants $\beta \in \left(0, \frac12\right)$ and $C > 0$ such that
		$$\P\{\forall x \in \R \ \forall t \geq 0 \ |b(x, t)| \leq C(1 + |x|^\beta)\} = 1;$$
		\item $X(t) \to +\infty$ a.s.
	\end{itemize}
	Then almost surely
	$$X(t) \sim At, \ t \to \infty.$$
\end{theorem}
\begin{remark}
	Here, $X(t)$ is a weak solution, not necessarily unique.
\end{remark}
\begin{proof}
	Consider SDE (\ref{general_sde}) in the integral form:
	$$X(t) = x_0 + \int_0^t a(X(s), s)\d s + \int_0^t b(X(s), s)\d B(s).$$
	Estimate the expectation of $X^2(t)$:
	$$\E X^2(t) = \E \left(\left(x_0 + \int_0^t a(X(s), s)\d s\right) + \int_0^t b(X(s), s)\d B(s)\right)^2 \leq$$
	$$\leq 2\E\left(x_0 + \int_0^t a(X(s), s)\d s\right)^2 + 2\E \left(\int_0^t b(X(s), s)\d B(s)\right)^2.$$
	Since $a$ is bounded, for the first expectation we have the estimate
	$$\E\left(x_0 + \int_0^t a(X(s), s)\d s\right)^2 \leq (C_1 t + C_2)^2$$
	for some $C_1 > 0, \ C_2 > 0$.
	Estimate the second expectation using the Itô isometry and Jensen's inequality:
	$$\E \left(\int_0^t b(X(s), s)\d B(s)\right)^2 \leq C_3t + C_3\int_0^t(\E X^2(s))^\beta\d s$$
	for some $C_3 > 0$.
	Hence, we obtain the estimate
	$$\E X^2(t) \leq (C_4 t + C_5)^2 + C_6 \int_0^t (\E X^2(s))^\beta\d s$$
	for some $C_4, C_5, C_6 > 0$, or denoting $u(t) := \E X^2(t)$,
	$$u(t) \leq (C_4 t + C_5)^2 + C_6 \int_0^t u^\beta(s) \d s.$$
	Using a generalization of Grönwall's inequality (Lemma \ref{lemma_gronwall} in Appendix), we get
	$$u(t) \leq \tilde C \left((1 - \beta)t + (C_4 t + C_5)^{2 - 2\beta}\right)^\frac{1}{1 - \beta}$$
	for some $\tilde C > 0$.
	So
	$$\limsup_{t \to \infty} \frac{u(t)}{t^2} \leq C_7$$
	for some $C_7 > 0$.
	From the last inequality and local boundedness of the function $u$, it follows that
	$$u(t) = \E X^2(t) \leq C_8 (t^2 + 1)$$
	for some $C_8 > 0$.
	Then using Jensen's inequality,
	$$\E b^2(X(t), t) \leq \left(C\left(1 + \E|X(t)|^\beta\right)\right)^2 \leq 2C^2\left(1 + \left(\E|X(t)|^\beta\right)^2\right) \leq$$
	$$\leq 2C^2\left(1 + (\E X^2(t))^\beta\right) \leq 2C^2\left(1 + (C_8 (t^2 + 1))^\beta\right) \leq C_9 (1 + t^{2\beta})$$
	for some $C_9 > 0$.
	
	Hence, by Theorem \ref{lemma_integral}, almost surely
	$$\frac1t \int_0^t b(X(s), s)\d B(s) \to 0, \ t \to \infty.$$
	Therefore, almost surely for large $t$,
	$$\frac{X(t)}t = \frac{x_0}t + \int_0^t \frac{a(X(s), s)}s \d s + \frac1t \int_0^t b(X(s), s)\d B(s).$$
	Going to the limit as $t \to \infty$, we get almost surely
	$$\lim_{t \to \infty}\frac{X(t)}t = A$$
	and the theorem is proved.
\end{proof}

The next theorem states that the previous one holds if the coefficients have those properties for large $x$.

\begin{theorem}
Suppose that:
\begin{itemize}
	\item the coefficient $a = a(x, t)$ is bounded for $x \geq x_* > 0$, $t \geq 0$ and there exists the limit
	$$\lim_{x, t \to \infty} a(x, t) = A \ \text{a.s.},$$
	where $A > 0$ is a positive random variable;
	\item there exist $\beta \in \left(0, \frac12\right)$, $C > 0$, and $x_* > 0$ such that almost surely for any $x \geq x_* > 0$ and $t \geq 0$,
	$$|b(x, t)| \leq Cx^\beta;$$
	\item $X(t) \to +\infty$ a.s.
\end{itemize}
Then almost surely
$$X(t) \sim At, \ t \to \infty.$$
\end{theorem}
\begin{proof}
Construct a twice continuously differentiable function $f$ such that
$$f(x) = \begin{cases}
	0, \ x < x_*, \\
	x, \ x > 2x_*.
\end{cases}$$

Apply Itô's lemma to the process $\tilde X(t) := f(X(t))$:
$$\d \tilde X(t) = \tilde a(t)\d t + \tilde b(t) \d B(t),$$
where
$$\tilde a(t) := a(X(t), t)f'(X(t)) + \frac12 b^2(X(t), t)f''(X(t)),$$
$$\tilde b(t) := b(X(t), t)f'(X(t)).$$

Considering each of the cases $X(t) < x_*$, $x_* \leq X(t) \leq 2x_*$, and $X(t) > 2x_*$, it is easy to see that $\tilde a$ is bounded, almost surely
$$\lim_{t \to \infty} \tilde a(t) = \lim_{t \to \infty} a(X(t), t) = A$$
and almost surely
$$|\tilde b(t)| \leq \const (1 + \tilde X^\beta(t)).$$

It is clear that $X(t) \to \infty, \ t \to \infty,$ a.s. iff  $\tilde X(t) \to \infty, \ t \to \infty,$ a.s.

Applying the previous theorem to $\tilde X$, we obtain $\tilde X(t) \sim At, \ t \to \infty,$ a.s., which implies that $X(t) \sim At, \ t \to \infty,$ a.s.
\end{proof}

The following theorem lets to find the asymptotics of the solution in the case when the drift coefficient has power growth.

\begin{theorem}\label{theorem_power_asymptotics}
Suppose that:
\begin{enumerate}[label=(\arabic*)]
\item there exist a positive random variable $A > 0$ and a constant $\alpha \in (-1, 1)$ such that almost surely
$$a(x, t) \sim Ax^\alpha, \ x, t \to \infty;$$
\item there exist constants $A_1 \geq 0$ and $x_* > 0$ such that 
$$a(x, t) \leq A_1x^\alpha, \ x > x_*, \ t \geq 0;$$
\item there exist constants $C > 0$, $\beta \in (0, \frac{\alpha + 1}2)$, $x_* > 0$ such that almost surely for any $x \geq x_*$, $t \geq 0$,
$$|b(x, t)| \leq Cx^\beta;$$
\item $X(t) \to +\infty, \ t \to \infty,$ a.s.
\end{enumerate}
Then almost surely
$$X(t) \sim \left((1 - \alpha)At\right)^\frac1{1 - \alpha}.$$
\end{theorem}

\begin{remark}
Notice that the asymptotics of $X$ is nothing but the asymptotics of the following ODE solution $x$:
$$\d x(t) = Ax^\alpha(t) \d t.$$
\end{remark}

\begin{proof}
Construct a twice continuously differentiable function $f$ such that
$$f(x) = \frac{x^{1 - \alpha}}{1 - \alpha}, \ x \geq x_*.$$

Apply Itô lemma to the process $\tilde X(t) := f(X(t))$:
$$\d \tilde X(t) = \tilde a(t)\d t + \tilde b(t)\d B(t),$$
where
$$\tilde a(t) := a(X(t), t)f'(X(t)) + \frac12 b^2(X(t), t)f''(X(t)),$$
$$\tilde b(t) := b(X(t), t)f'(X(t)).$$

For large $t$,
$$\tilde a(t) = \frac{a(X(t), t)}{X^\alpha(t)} - \frac\alpha2\frac{b^2(X(t), t)}{X^{\alpha + 1}(t)}\d t,$$
$$\tilde b(t) = \frac{b(X(t), t)}{X^\alpha(t)}.$$
We have almost surely
$$\lim_{t \to \infty} \frac{a(X(t), t)}{X^\alpha(t)} = \lim_{t \to \infty} \frac{AX^\alpha(t)}{X^\alpha(t)} = A$$
and
$$\lim_{t \to \infty} \frac{b^2(X(t), t)}{X^{\alpha + 1}(t)} \leq \lim_{t \to \infty}\frac{C^2 X^{2\beta}(t)}{X^{\alpha + 1}(t)} = C^2\lim_{t \to \infty} \frac1{X^{\alpha - 2\beta + 1}(t)} = 0$$
since $\alpha - 2\beta + 1 > 0$ and $X(t) \to \infty, \ t \to \infty,$ a.s. by condition.
Hence, almost surely
$$\lim_{t \to \infty} \tilde a(t) = A.$$

Denote $\tau := \inf\{t \geq 0\colon \forall s \geq t \ X(s) \geq x_*\}$. For $t \geq \tau$
$$|\tilde b(t)| = \frac{|b(X(t), t)|}{X^\alpha(t)} \leq C\frac{X^\beta(t)}{X^\alpha(t)} = CX^{\beta - \alpha}(t) = C\left((A(1 - \alpha)\tilde X(t))^\frac1{1 - \alpha}\right)^{\beta - \alpha} = $$
$$= \const \cdot (\tilde X(t))^\frac{\beta - \alpha}{1 - \alpha} =: \const \cdot (\tilde X(t))^{\tilde \beta}.$$

Hence, we have the following equivalences:

$$\tilde \beta < \frac12 \Leftrightarrow \frac{\beta - \alpha}{1 - \alpha} < \frac12 \Leftrightarrow 2\beta - 2\alpha < 1 - \alpha \Leftrightarrow \beta < \frac{\alpha + 1}2.$$

It is clear that $X(t) \to \infty, \ t \to \infty,$ a.s. iff $\tilde X(t) \to \infty, \ t \to \infty,$ a.s. Applying the previous theorem to the process $\tilde X$, we obtain $\tilde X(t) \sim At, \ t \to \infty,$ a.s., i.e., $f(X(t)) \sim At, \ t \to \infty,$ a.s.

For sufficiently large $t$, there exists the inverse function $f^{-1}(t)$, which is a power function. Applying $f^{-1}$ to both parts of the last equivalence, we obtain almost surely
$$f^{-1}(f(X(t))) \sim f^{-1}(t), \ t \to \infty,$$
i.e., almost surely
$$X(t) \sim \left((1 - \alpha)At\right)^\frac1{1 - \alpha}, \ t \to \infty.$$
\end{proof}
\begin{remark}
	Instead of (1), consider the following condition:
	\begin{enumerate}[label=(1')]
		\item there exist a positive random variable $A_2 > 0$ and a constant $\alpha \in (-1, 1)$ such that almost surely
		$$a(x, t) \geq A_2x^\alpha, \ x \geq x_* > 0, t \geq 0.$$
	\end{enumerate}
	If conditions (1'), (2), (3), and (4) hold then one can show (similarly to the proof of the previous theorem) that
	$$\liminf_{t \to \infty} \frac{X(t)}{t^\frac1{1 - \alpha}} \geqslant \left((1 - \alpha)A_2\right)^\frac1{1 - \alpha}.$$
\end{remark}

\section{Transience of Solutions}

Consider the system (\ref{sde_radius_new})-(\ref{sde_angle_new}) again:
\begin{eqnarray}
	\label{sde_radius_new} \d R(t) = \mu(R(t), \Phi(t))\d t + \sigma(R(t), \Phi(t))\d W^{(1)}(t), \ R(0) = r_0, \\
	\label{sde_angle_new} \d\Phi(t) = \nu(R(t), \Phi(t))\d t + \chi(R(t), \Phi(t))\d W(t), \ \Phi(0) = \varphi_0,
\end{eqnarray}
where coefficients $\mu, \sigma, \nu, \chi$ are arbitrary (not related to the coefficients $a, b$ of the initial SDE (\ref{sde})).

Define the next operators for the radius SDE (\ref{sde_radius}):
$$L_\varphi[V](r) := \mu(r, \varphi)V'(r) + \frac12\sigma^2(r, \varphi)V''(r), \ r > 0, \ \varphi \in \R^n, \ V \in \C^2(0, \infty),$$
where $\C^2(0, \infty)$ is the set of all twice continuously differentiable function on $(0, \infty)$.

\begin{theorem}\label{theorem_transiency}
Suppose that:
\begin{enumerate}
\item for any starting point $(r_0, \varphi_0)$, $r_0 \neq 0$, there exists a unique solution $(R, \Phi)$ of the system (\ref{sde_radius_new})-(\ref{sde_angle_new}), which is a strong Markov process;
\item $\mu, \sigma$ are continuous, $\sigma(r, \varphi) \geq \sigma_* > 0, \ r > 0, \ \varphi \in \R^n$;
\item there exist a non-decreasing function $V_0$ and $\delta > 0$ such that $V_0(0) = -\infty$ and
$$\forall r \in (0, \delta) \ \forall \varphi \in \R^n \ L_\varphi[V_0](r) \leqslant 0;$$
\item there exist a decreasing function $V_\infty$ and a constant $\Delta > \delta$ such that $|V_\infty(\infty)| < \infty$ and 
$$\forall r > \Delta \ \forall \varphi \in \R^n \ L_\varphi[V_\infty](r) \leqslant 0.$$
\end{enumerate}
Then almost surely:
\begin{enumerate}
\item $R(t) > 0, \ t \geqslant 0$;
\item $R(t) \to \infty, \ t \to \infty$.
\end{enumerate}
\end{theorem}
\begin{remark}
If coefficients $\mu, \sigma, \nu, \chi$ are locally Lipschitz and have linear growth at infinity then it is known that the SDE solution exists, is unique, and is strong Markov process (see § 10 of Part I in \cite{gichman}).
\end{remark}
\begin{remark}
The third condition implies the first statement of the theorem. The fourth condition guarantees that being far away from the origin, the process $R$ goes to infinity with a high probability.
\end{remark}
\begin{remark}
Notice that there are no requirements for behavior of the generator in the interval $r \in [\delta, \Delta]$. It may be possible to find a common function $V$ in the interval $(0, \infty)$ instead of two functions $V_0$ and $V_\infty$ such that $V(0) = -\infty$, $V(\infty) \in \R$, and
$$\forall r > 0 \ \forall \varphi \in \R^n \ L_\varphi[V](r) \leqslant 0.$$
\end{remark}
\begin{proof}
Notice that since coefficients $\mu, \sigma$ are continuous, they are bounded on compact sets, hence, by Corollary \ref{corollary_exit}, the process $R$ exits any interval $[a, b] \subset (0, \infty)$ almost surely. Define $\tau_r := \inf\{t \geqslant 0\colon R(t) = r\}, \ r \geqslant 0$.
	
\emph{Step 1}. Suppose first that $R(0) = r_0 \in (0, \delta)$. Let $\varepsilon \in (0, r_0)$. Since the solution exits any interval a.s., $\tau_\varepsilon \wedge \tau_\delta < \infty$ a.s. Then by Lemma \ref{generalized_lemma_s}, we have
$$\P\{\tau_\delta < \tau_\varepsilon\} \geqslant \frac{V_0(r_0) - V_0(\varepsilon)}{V_0(\delta) - V_0(\varepsilon)} \to 1, \ \varepsilon \to 0.$$
Therefore, $\tau_\delta < \infty$ a.s. and $\P\{\tau_0 < \tau_\delta\} = 0$. This implies (by virtue of continuity of $R$) that $\P\{R(t) > 0, \ t \geqslant 0\} = 1$. 

\emph{Step 2}. By the strong Markov property, the distribution of the process\footnote{By $R_\delta$ we denote the solution of the corresponding SDE with the starting point $R(0) = \delta$. Similarly we define $\Phi_\xi$.} $(R(\tau_\delta + t), \Phi(\tau_\delta + t))_{r \geqslant 0}$ is the same as the distribution of the process $(R_\delta(t), \Phi_\xi(t))$, where $\xi \sim \Phi(\tau_\delta)$, $\xi$ is independent of $\Phi(\tau_\delta)$. Therefore, without loss of generality, we suppose now that $R(0) = \delta$. Let $\tilde \Delta > \Delta$. Since $\mu$ and $\sigma$ are continuos and $\sigma > 0$, $\mu$ and $\sigma$ are bounded for $r \in [\delta / 2, \tilde\Delta]$ ($|\mu| \leq M$, $\sigma \leq \Sigma$ for some $M, \ \Sigma > 0$). Choose some decreasing function $V$ such that $L_\varphi[V](r) \leq 0$ for $r \in [\delta / 2, \tilde\Delta]$ and $\varphi \in \R^n$ (e.g., $V(r) = e^{-\frac{2M}{\Sigma^2}r}$). Since the solution exits any interval a.s., $\tau_{\tilde\Delta} \wedge \tau_{\delta/2} < \infty$ a.s. Then by Lemma \ref{generalized_lemma_s},
$$\P_{\delta}\left\{\tau_{\tilde\Delta} < \tau_{\delta/2}\right\} \geqslant \frac{V(\delta) - V(\delta/2)}{V(\tilde\Delta) - V(\delta/2)} =: p > 0,$$
$$\P_{\delta}\left\{\tau_{\delta/2} < \tau_{\tilde\Delta}\right\} \geqslant \frac{V(\tilde\Delta) - V(\delta)}{V(\tilde\Delta) - V(\delta/2)} = 1 - p < 1.$$
Using the strong Markov property $k$ times, one can show that the probability of exiting the interval $\left[\frac\delta2, \tilde\Delta\right]$ $k$ times from the left end and return into the interval $[\delta, \tilde\Delta]$ is not greater than $(1 - p)^k \to 0, \ k \to \infty$. Thus, $\tau_{\tilde\Delta} < \infty$ a.s.

\emph{Step 3}. By the strong Markov property, the distribution of the process $((R(\tau_{\tilde\Delta} + t), \Phi(\tau_{\tilde\Delta} + t)))_t$ is the same as the distribution of the process $((R_{\tilde\Delta}(t), \Phi_\xi(t)))_{t \geqslant 0}$, where $\xi \sim \Phi(\tau)$ and $\xi$ is independent of $\Phi(\tau)$. Therefore, without loss of generality, we suppose now that $R(0) = \tilde\Delta$. Let $L > \tilde\Delta$ and $\Delta^* \in (\Delta, \tilde\Delta)$. Since the solution exits any interval, $\tau_L \wedge \tau_{\Delta^*} < \infty$ a.s. By Lemma \ref{generalized_lemma_s},
$$\P_{\tilde\Delta}\left\{\tau_{\Delta^*} < \tau_L\right\} \leq \frac{V_\infty(L) - V_\infty(\tilde\Delta)}{V_\infty(L) - V_\infty(\Delta^*)}.$$
As $L \to \infty$, we obtain
$$\P_{\tilde\Delta}\left\{\inf_{t \geq 0}R(t) \leq \Delta^*\right\} \leq \frac{V_\infty(\infty) - V_\infty(\tilde\Delta)}{V_\infty(\infty) - V_\infty(\Delta^*)}.$$
Then we have the following estimates:
$$\P_{\tilde\Delta}\left\{\liminf_{t \to \infty}R(t) \leq \Delta^*\right\} \leq \P_{\tilde\Delta}\left\{\inf_{t \geq 0}R(t) \leq \Delta^*\right\} \leq  \frac{V_\infty(\infty) - V_\infty(\tilde\Delta)}{V_\infty(\infty) - V_\infty(\Delta^*)} = $$
$$= 1 - \frac{V_\infty(\tilde\Delta) - V_\infty(\Delta^*)}{V_\infty(\infty) - V_\infty(\Delta^*)} =: p_{\tilde\Delta}.$$
Notice that $p_{\tilde\Delta} \to 0$ as $\tilde\Delta \to \infty$. As $\tilde\Delta \to \infty$, we obtain
$$\P\left\{\liminf_{t \to \infty}R(t) \leq \Delta^*\right\} = 0.$$
As $\Delta^* \to \infty$, we obtain
$$\P\left\{\liminf_{t \to \infty}R(t) < +\infty\right\} = 0 \Rightarrow \P\left\{\liminf_{t \to \infty}R(t) = +\infty\right\} = 1 \Rightarrow$$
$$\Rightarrow \P\left\{\lim_{t \to \infty}R(t) = +\infty\right\} = 1.$$
\end{proof}

\begin{example}
Consider the following system of SDEs:
$$\d X_1(t) = a_1(X_1(t), X_2(t))\d t + \d W_1(t),$$
$$\d X_2(t) = a_2(X_1(t), X_2(t))\d t + \d W_2(t),$$
where $a_i, i \in \{1, 2\},$ are continuous and for large $x$,
$$\lim_{x_i \to \infty}\frac{a_i(x_1, x_2)}{\sqrt{|x_i|}} \sgn x_i  \geqslant 1, \ i \in \{1, 2\},$$

Applying Theorem \ref{theorem_transiency} to this system and setting $V_0(r) := \ln r$, $V_\infty(r) =: \frac1r$, we see that $|X(t)| \to \infty, \ t \to \infty,$ a.s.
\end{example}

\section{Stabilization of the Angle}

Consider SDE (\ref{sde_angle_new}) for $\Phi$.

\begin{theorem}\label{theorem_angle_stabilization}
Suppose that:
\begin{enumerate}
\item $\P\{\forall t \geqslant 0 \ R(t) > 0\} = 1$;
\item $\liminf_{t \to \infty} \frac{R(t)}{t^\gamma} \geqslant C^*$ for some random variable $C^* > 0$ and some non-random constant $\gamma > 0$;
\item $|\nu(r, \varphi)| \leqslant \frac {\mu^*}{r^{\delta_1}}, |\chi(r, \varphi)| \leqslant \frac {\chi^*}{r^{\delta_2}}$ for all $\varphi \in \R^n$ and large $r$, where $\nu^*, \chi^* > 0, \delta_1 > \frac1\gamma, \delta_2 > \frac1{2\gamma}$.
\end{enumerate}
Then there exists the limit $\lim_{t \to \infty} \Phi(t)$ a.s.
\end{theorem}

\begin{proof}
Rewrite SDE (\ref{sde_angle_new}) in the integral form:
$$\Phi(t) = \varphi_0 + \int_0^t \nu(R(s), \Phi(s))\d s + \int_0^t \chi(R(s), \Phi(s))\d W(s).$$
Then the limit
$$\Phi_\infty = \varphi_0 + \int_0^\infty \nu(R(t), \Phi(t))\d t + \int_0^\infty \chi(R(t), \Phi(t))\d W(t)$$
exist a.s. if the integrals in the right-hand side are convergent.

To prove convergence of the first and the second integrals, use conditions 2 and 3 of the theorem:
$$\delta_1 > \frac1\gamma \Rightarrow \delta_1\gamma > 1 \Rightarrow \int_1^\infty\frac{\d t}{t^{\delta_1\gamma}} < \infty \Rightarrow \int_0^\infty \left|\nu(R(t), \Phi(t))\right|\d t < \infty \ \text{a.s.};$$
$$\delta_2 > \frac1{2\gamma} \Rightarrow 2\delta_2\gamma > 1 \Rightarrow \int_1^\infty\frac{\d t}{t^{2\delta_2\gamma}}\d t < \infty \Rightarrow \int_0^\infty \chi^2(R(t), \Phi(t))\d t < \infty \ \text{a.s.} \Rightarrow$$
$$\Rightarrow \int_0^\infty \chi(R(t), \Phi(t))\d W(t) \ \text{is well-defined}.$$
\end{proof}

\section{Asymptotics of the Radius}

\begin{theorem}
Consider SDE (\ref{sde_radius_new}) for the radius:
$$\d R(t) = \mu(R(t), \Phi(t))\d t + \sigma(R(t), \Phi(t))\d W^{(1)}(t).$$

Suppose that the following conditions hold:
\begin{enumerate}
\item there exist a continuous bounded function $M\colon \R^n \to (0, \infty)$ and a constant $\alpha \in (-1, 1)$ such that for any $\varphi \in \R^n$,
$$\mu(r, \varphi) \sim M(\varphi)r^\alpha, \ r \to \infty;$$
\item there exist constants $C > 0$, $\beta \in (0, \frac{\alpha + 1}2)$, $r_* > 0$ such that for any $r \geq r_*$, $\varphi \in \R^n$,
$$|\sigma(r, \varphi)| \leq Cr^\beta;$$
\item $R(t) \to \infty, \ t \to \infty,$ a.s.;
\item $\exists \Phi_\infty := \lim_{t \to \infty} \Phi(t)$ a.s.
\end{enumerate}
Then
$$R(t) \sim \left((1 - \alpha)M(\Phi_\infty)t\right)^\frac1{1 - \alpha}, \ t \to \infty.$$
\end{theorem}

\begin{proof}
Notice that $\mu(r, \Phi(t)) \sim M(\Phi_\infty)r^\alpha, \ r, t \to \infty$. Applying Theorem \ref{theorem_power_asymptotics} to the process $R$, we obtain the statement of the theorem. 
\end{proof}

\begin{example}
Consider the following $n$-dimensional ($n \geq 2$) SDE:
$$\d X(t) = |X(t)|^{\alpha - 1}X(t)\d t + \d W(t), \ X(0) = x_0 \neq 0,$$
where $-1 < \alpha < 1$, $W$ is an $n$-dimensional Wiener process.

If for $x \in \R^n$ and $-1 < \alpha < 1$ we denote $x^\alpha := |x|^{\alpha - 1}x$, then the previous SDE can be written in the following form:
$$\d X(t) = X^\alpha(t)\d t + \d W(t).$$

Using this article's results, let's prove that almost surely:
\begin{itemize}
\item $\forall t \geq 0 \ X(t) \neq 0$;
\item $|X(t)| \to \infty, \ t \to \infty$;
\item $\exists \lim_{t \to \infty} \frac{X(t)}{|X(t)|}$.
\end{itemize}
\end{example}
\begin{proof}
One can obtain the SDE for the radius $R = |X|$ of the process $X$ using the general formula given in Section 1:
$$\d R(t) = \left(R^\alpha(t) + \frac{n - 1}{2R(t)}\right)\d t + \d W^{(1)}(t),$$
where $W^{(1)}$ is some one-dimensional Wiener process.

The scale function for this SDE is
$$s(r) = \int_1^r \frac1{u^{n - 1}} \exp \frac{2(1 - u^{\alpha + 1})}{\alpha + 1}\d u.$$
Since $|s(0)| = \infty$ and $|s(+\infty)| < \infty$, classic results imply that the process $R$ never hits zero and $R(t) \to \infty, \ t \to \infty,$ almost surely. This means that the process $X$ never hits the origin and $|X(t)| \to \infty, \ t \to \infty,$ almost surely. 

Let's find the asymptotics of the process $|X|$ using Theorem \ref{theorem_power_asymptotics} being applied to the SDE for $R = |X|$. The drift coefficient $\mu(r) = r^\alpha + \frac{n - 1}{2r} \sim r^\alpha, \ t \to \infty$ (here $A = 1, \ \alpha = \alpha$), the diffusion coefficient $\sigma(r) = 1 \leq r^0$ (here $C = 1, \ \beta = 0$) and a.s. $|X(t)| \to \infty, \ t \to \infty$. Then by Theorem \ref{theorem_power_asymptotics} almost surely
$$|X(t)| \sim ((1 - \alpha)t)^\frac1{1 - \alpha}, \ t \to \infty.$$

Let's prove that the angle $\frac X{|X|}$ stabilizes using Theorem \ref{theorem_angle_stabilization}. From the general SDE for the angle (see Section 1) one can obtain the SDE for the angle in our case:
$$\d \frac{X(t)}{|X(t)|} = -\frac{2I\tangential(X(t)) + n - 1}{2|X(t)|^3}X(t)\d t + \frac{I\tangential(X(t))}{|X(t)|}\d W(t),$$
where $I$ is a unit $n \times n$ matrix. The first condition of the theorem is satisfied. The second condition (existence of the lower asymptotics) follows from existence of the exact asymptotics (here $\gamma = \frac1{1 - \alpha}$). Check the third condition (find estimates for the coefficients):
$$\left|-\frac{2I\tangential(x) + n - 1}{2|x|^3}x\right| \leq \frac{2\sqrt{n - 1} + n - 1}{2|x|^2},$$
$$\left|\frac{I\tangential(x)}{|x|}\right| = \frac{\sqrt{n - 1}}{|x|},$$
i.e., $\delta_1 = 2, \ \delta_2 = 1$.
The third condition of the theorem is satisfied since for such $\delta_1, \delta_2$, and $\gamma = \frac1{1 - \alpha}$,
$$\begin{cases}
	\delta_1 > \frac1\gamma,\\
	\delta_2 > \frac1{2\gamma}
\end{cases} \Leftrightarrow \alpha > -1.$$
Hence, by Theorem \ref{theorem_angle_stabilization}, there exists a limit $\lim_{t \to \infty} \frac{X(t)}{|X(t)|}$ almost surely.
\end{proof}

\begin{example}
Let's perturb the drift coefficient of the SDE from the previous example:
$$\d X(t) = (X^\alpha(t) + f(X(t)))\d t + \d W(t),$$
where the function $f$ is such that:
\begin{itemize}
\item $|f\radial(x)| = o\left(\frac1r\right), \ |x| \to 0, \qquad |f\radial(x)| = o(|x|^{\alpha}), \ |x| \to \infty$;
\item $|f\tangential(x)| \leq C_2|x|^{\alpha - \varepsilon}$ for large $|x|$, where $C_2 > 0$, $\varepsilon \in (0, 1 + \alpha)$.
\end{itemize}
Check that the solution of this SDE has the same properties as one from the previous example.
\end{example}
\begin{proof}
Notice that 
$$|f\radial(x)| = o(|x|^{\alpha}), \ |x| \to \infty \qquad \Rightarrow \qquad |\langle f(x), x \rangle| = o(|x|^{1 + \alpha}), \ |x| \to \infty.$$
	
The SDE for the radius $|X|$ has the form
$$\d |X(t)| = \left(|X(t)|^\alpha + \frac{\langle f(X(t)), X(t)\rangle}{|X(t)|} + \frac{n - 1}{2|X(t)|}\right)\d t + \d W^{1}(t).$$

Use Theorem \ref{theorem_transiency} to prove that the solution does not hit zero and goes to infinity. Consider a Lyapunov function $V_0(r) := -\frac1{r^{n - 1}}$, which is increasing, $V_0(0) = 0$, and a Lyapunov function $V_\infty(r) := \frac1r$, which is decreasing, $|V_\infty(+\infty)| < \infty$.
We have:
$$LV_0(r) \leq \frac{n - 1}{r^{n - \alpha}} - \frac{n - 1}{2r^{n + 1}} + o\left(\frac1{r^{n + 1}}\right) \leq 0$$
for $r \to 0$,
$$LV_\infty(r) \leq -\frac1{r^{2 - \alpha}} - \frac{n - 3}{2r^3} + o(r^{\alpha - 2})\leq 0$$
for $r \to \infty$.
Hence,
$$\P\{X(t) \neq 0, \ t \geq 0\} = 1, \qquad \P\{|X(t)| \to \infty, \ t \to \infty\} = 1.$$

Let's find the asymptotics of $|X|$. Since 
$$\frac{\langle f(r\varphi), r\varphi \rangle}{r} \leq \frac{C_1|r\varphi|^{1 + \alpha - \varepsilon}}{r} = Cr^{\alpha - \varepsilon},$$
the drift coefficient
$$\mu(r, \varphi) = r^\alpha + \frac{\langle f(r\varphi), r\varphi \rangle}{r} + \frac{n - 1}{2r} \sim r^\alpha, \ t \to \infty$$
(here $A = 1, \ \alpha = \alpha$), similarly to the previous example, by Theorem \ref{theorem_power_asymptotics} we have almost surely
$$|X(t)| \sim ((1 - \alpha)t)^\frac1{1 - \alpha}, \ t \to \infty.$$

Let's prove that the angle $\frac X{|X|}$ stabilizes. From the general SDE for the angle one can obtain the SDE for the angle in our case:
$$\d \frac{X(t)}{|X(t)|} = \left(\frac{f\tangential(X(t))}{|X(t)|} -\frac{2I\tangential(X(t)) + n - 1}{2|X(t)|^3}X(t)\right)\d t + \frac{I\tangential(X(t))}{|X(t)|}\d W(t),$$
Like in the previous example, the first and the second conditions of Theorem \ref{theorem_angle_stabilization} are satisfied. Check the third condition (find estimates for the coefficients):
$$\left|\frac{f\tangential(x)}{|x|}-\frac{2I\tangential(x) + n - 1}{2|x|^3}x\right| \leq \frac{C_2}{|x|^{1 - \alpha + \varepsilon}} + \frac{2\sqrt{n - 1} + n - 1}{2|x|^2} \leq \frac{C_3}{|x|^{1 - \alpha + \varepsilon}}$$
for large $|x|$, where $C_3 > 0$ (since $1 - \alpha + \varepsilon < 1$),
$$\left|\frac{I\tangential(x)}{|x|}\right| = \frac{\sqrt{n - 1}}{|x|},$$
i.e., $\delta_1 = 1 - \alpha + \varepsilon, \ \delta_2 = 1$.
The third condition is satisfied, because for such $\delta_1, \delta_2$, and $\gamma = \frac1{1 - \alpha}$,
$$\begin{cases}
	\delta_1 > \frac1\gamma,\\
	\delta_2 > \frac1{2\gamma}
\end{cases} \Leftrightarrow \begin{cases}
	\varepsilon > 0,\\
	\alpha > -1.
\end{cases}$$
Hence, by Theorem \ref{theorem_angle_stabilization}, there exists the limit $\lim_{t \to \infty} \frac{X(t)}{|X(t)|}$ almost surely.

\end{proof}

\section{Appendix. Auxiliary Results}

Consider the following generalization of Grönwall's inequality (see § 1.7 in \cite{mao}).

\begin{lemma}\label{lemma_gronwall}
	Let $u\colon \R_+ \to \R_+$ be a continuous function satisfying the inequality
	$$u(t) \leq a(t) + C\int_0^t u^\beta(s) \d s,$$
	where $C > 0, \ 0 < \beta < 1$, the function $a\colon \R_+ \to \R_+$ is non-decreasing and continuous. Then
	$$u(t) \leq \tilde C\left((1 - \beta)t + a^{1 - \beta}(t)\right)^\frac1{1 - \beta}, \qquad \ \text{where} \ \tilde C := C^\frac1{1 - \beta}.$$
\end{lemma}

Let $X$ be a solution of the following one-dimensional non-autonomous SDE:
\begin{equation}
	\label{general_sde} \d X(t) = a(X(t), t, \omega)\d t + b(X(t), t, \omega)\d W(t), \qquad X(0) = x_0 \in (x_1, x_2),
\end{equation}
where $a(x, t, \omega), b(x, t, \omega)$ are $\B(\R)\times[0, t]\times\F_t$-measurable. For this SDE, define a family of operators
\begin{equation}\label{operator}
	L_t[u](x) := a(x, t)u'(x) + \tfrac12 b^2(x, t)u''(x), \qquad x \in [x_1, x_2], \ t \geqslant 0,
\end{equation}
and the exit time
$$\tau := \inf \{t \geqslant 0\colon X(t) \notin (x_1, x_2)\}.$$

The following lemma allows to prove that under some conditions, an SDE (\ref{general_sde}) solution exits any interval $(x_1, x_2)$ after a finite time.

\begin{lemma}\label{generalized_lemma_u}
	Let the functions $a$ and $b$ be bounded on $[x_1, x_2]$. Suppose that there exists a non-random function $u$ such that
	$$L_t[u](x) \leqslant -1, \ x \in [x_1, x_2], \ t \geqslant 0.$$
	Then\footnote{Notation $\E_{x_0}$ and $\P_{x_0}$ emphasize that $X(0) = x_0$.} $\E_{x_0}\tau \leqslant 2\max_{x \in [x_1, x_2]}|u(x)|$. As a consequence, almost surely $\tau < \infty$.
\end{lemma}

The proof of the lemma in standard (e.g., see § 3.7 in \cite{khasminskii}).

\begin{corollary}\label{corollary_exit}
	Let the functions $a$ and $b$ be bounded on $[x_1, x_2]$ and $b > \delta > 0$ on $[x_1, x_2]$ for some $\delta > 0$. Then $\tau < \infty$ almost surely.
\end{corollary}
\begin{proof}
	From the conditions of the corollary, it follows that there exist constants $A \neq 0$ and $B > 0$ such that $a(\cdot, \cdot) \geqslant A, \ 0 < b(\cdot, \cdot) \leqslant B$. Let the function $u$ is such that $u' < 0$, $u'' > 0$. Then
	$$Au'(x) + \frac12 B^2u''(x) = -1 \Rightarrow L_\varphi[u](x) \leqslant -1.$$
	For example,
	$$u(x) := \exp\left(-\frac{2A}{B^2}x\right) - \frac{x}{A}, \ x \in [x_1, x_2].$$
	Hence, $\tau < \infty$ a.s. by Lemma \ref{generalized_lemma_u}.
\end{proof}

The next lemma allows to estimate probabilities of exiting through the left ot the right end of the interval $[x_1, x_2]$.

\begin{lemma}\label{generalized_lemma_s}
	Let the conditions of Lemma \ref{generalized_lemma_u} hold. Besides this, suppose that there exists a decreasing function $V$ such that
	$$L_t[V](x) \leqslant 0,\; x \in [x_1, x_2].$$
	Then
	$$\P_{x_0} \{X(\tau) = x_1\} \leqslant \frac{V(x_0) - V(x_2)}{V(x_1) - V(x_2)}, \qquad \P_{x_0} \{X(\tau) = x_2\} \geqslant \frac{V(x_1) - V(x_0)}{V(x_1) - V(x_2)}.$$
\end{lemma}

Proofs of these lemmas are standard and exploit Itô's lemma on the interval $[0, \tau]$ (for the proof ideas, see § 16 of Part I in \cite{gichman}).

\end{document}